\documentclass{IEEEtran}
\usepackage{amsmath,amsthm,amsfonts}

\newtheorem{definition}{Definition}
\newtheorem{theorem}{Theorem}
\newtheorem{lemma}[theorem]{Lemma}

\newtheorem{remark}{Remark}

\title{In theory, \textit{You crack under pressure!}}
\author{Elvis Dohmatob}

\begin{document}
\maketitle

\section{Life at the opera}
\textit{``It's opening night at the opera, and your friend
is the \textit{prima donna}
(the lead female singer). You will not be in the audience, but you
want to make sure she receives a standing ovation -- with every
audience member standing up and clapping their hands for her.''}
-- Google code jam \cite{codejam}.

\paragraph*{Notation} $\mathbb{N} = \{0, 1, 2, ...\}$ denotes the
  natural numbers and $[k] := [0, k] \cap \mathbb{N} = \{0, 1, 2, ...,
  k\}$, for every $k \in \mathbb{N}$.

\begin{definition}[Shyness \cite{codejam}]
  A spectator is said to have shyness level $s \in \mathbb{N}$ if they only
stand up for an applause when $s$ or more spectators are already up
and applauding. Let $p_s \in \mathbb{N}$ be the
number  spectators with shyness level $s$.
\end{definition}

\paragraph*{Problem statement}
Given a prescribed distribution $(p_0, p_1, ...., p_k)$ of shyness in
the audience,  with $p_k \ne 0$ (i.e $k$ is the shyness level of the
shyest spectators), invite as few friends of the prima dona as
possible (with their shyness levels distributed as you wish), say
$r(p_0, p_1, ..., p_k)$ friends, so that at the end she receives a
standing  ovation
\cite{codejam}.

\begin{remark} The natural number $r(p_0, p_1, ..., p_k)$ is
  well-defined, thanks to the well-ordering principle.
\end{remark}

\subsection{Some examples to warm up.}
The following examples are taken from the reference \cite{codejam}.
\begin{itemize}
\item[1.] $r(1, 1, ..., 1) = 0$. This is because the audience will
  eventually produce a standing ovation on its own.
\item[2.] $r(0, 9) = 1$. Inviting a \textit{bold} friend is optimal.
\item[3.] $r(1, 1, 0, 0, 1, 1) = 2$. Inviting two friends with shyness
  level $2$ is optimal.
\end{itemize}

\section{Solution: A short reliable program}
We will prove that: \textit{(a)} $r(p_0, p_1, ..., p_k) \le k$; and
\textit{(b)} $r(p_0, p_1, ..., p_k)$ is computable in
linear time $\mathcal{O}(k)$. Moreover, the proof will be
constructive, producing an algorithm which effectively computes
$r(p_0, p_1, p_2, ..., p_k)$ in $k$ steps.

\subsection{Preliminaries}
\begin{definition}[Insolubility of shyness levels]
Given a shyness level $s \in [k]$, the audience is said
to be $s$-insoluble iff there is a shyness level $s' \in [s]$ such
that $\sum_{j \in [s'-1]}p_j < s'$. Otherwise, we say
the audience is $s$-soluble.
\end{definition}

The idea behind insolubility is the following. An audience which
is  $s$-insoluble contains less than $s$ spectators who have
shyness less than $s$. Therefore these guys will never stand up, thus
blocking the guys with shyness level $s$. In particular, there won't
be a standing ovation for the prima dona.

The following Lemma gives a powerful necessary
and sufficient condition for a standing ovation to eventually occur.
\begin{lemma}
  There will eventually be a standing ovation iff the audience is
  $k$-soluble.
\label{thm:xtic}
\end{lemma}

\begin{proof} Indeed, ``there is eventually a standing ovation'' iff
  ``for every shyness level $s \in [k]$ there are at least
  $s$ spectators with shyness level less than $s$'' iff ``$\sum_{j \in
    [s-1]}p_j \ge s$, $\forall s \in [k]$'' iff ``the
  audience is $k$-soluble''.
\end{proof}

\subsection{The program proper}
Consider the following short program:
\begin{itemize}
\item[1.] \textbf{INITIALIZE} $s \leftarrow 1$, $r \leftarrow 0$.
\item[2.] \textbf{CHECK} If $\sum_{j \in [s-1]}p_j < s$, then
  \begin{itemize}
    \item[3.] \textbf{INVITE} a friend with any shyness level $s' \in [s]$.
    \item[4.] \textbf{UPDATE} $p_{s'} \leftarrow p_{s'} + 1$, $r \leftarrow r + 1$.

  \end{itemize}
\item[5.] \textbf{UPDATE} $s \leftarrow s + 1$.
\item[6.] \textbf{CHECK} If $s = k$, then \textbf{RETURN} $r$. Else
  \textbf{GOTO} 2.
\end{itemize}

\begin{theorem}
The above program terminates after exactly $k$ steps.
Once it terminates, the resulting audience is $k$-soluble, and thus
there will eventually be a standing ovation. Moreover, the program
outputes the least number of friends to invite, namely $r(p_0, p_1,
..., p_k)$.
\label{thm:main_thm}
\end{theorem}

We will need the following useful Lemmas for the proof.
\begin{lemma}
If the audience is $s$-soluble but $(s+1)$-insoluble, then
it becomes $(s+1)$-soluble upon the invitation of a friend with any
shyness level less than $s + 1$.
\label{thm:small_lemma}
\end{lemma}

\begin{proof}
Straightforward. Nothing to do.
\end{proof}

\begin{lemma}
  Define $0 < s_0$ := least $s \in [k]$ s.t. the audience is
  $s$-insoluble ($s_0 := \infty$ if no such $s$ exists).
  Then it holds that
  \begin{eqnarray}
    \label{eq:formula}
    \left .
    \begin{aligned}
      r(p_0,p_1,...,p_k) = 0 \text{ if }s_0 = \infty, \text{ and }
      r(p_0,p_1,...,p_k) = \\ 1 + r(p_0,...,p_{s' - 1}, p_{s'} + 1, p_{s' + 1}, ...,
      p_k) \forall s' \in [s_0-1], \text { else}.
      \end{aligned}
    \right\}
    \end{eqnarray}
  \label{thm:formula}
\end{lemma}

\begin{proof}
Indeed if $s_0 < \infty$, then by Lemma \ref{thm:small_lemma}
inviting a friend with any shyness $s' \in [s_0-1]$ will
simply subtract $1$ from the least number of friends required to produce a
standing ovation. This proves the first part of formula
\eqref{eq:formula}. On the other hand, if $s_0 = \infty$, then the
audience is $k$-soluble, and thus (Lemma \ref{thm:xtic}) will
eventually produce a standing ovation on its own.

\end{proof}

\begin{proof}[Proof of Theorem \ref{thm:main_thm}]
Indeed, the program does nothing but compute $r(p_0,p_1, ..., p_k)$
via the formula \eqref{eq:formula} established in Lemma
\ref{thm:formula}, word-for-word. Also, by construction, it halts
after exactly $k$ steps and its output $r$ is at most $k$. We are
done.
\end{proof}

\paragraph{Example attaining the bound} $r(0, 0, ...., 0, p_k) = k$.

\small
\bibliographystyle{splncs03}
\bibliography{bib}

\begin{thebibliography}{1}
\providecommand{\url}[1]{\texttt{#1}}
\providecommand{\urlprefix}{URL }

\bibitem{codejam}
Google: Standing {O}vation, {P}roblem {A}, {G}oogle code jam.  (2015),
  \url{http://code.google.com/codejam/contest/6224486/dashboard}

\end{thebibliography}

\end{document}